\documentclass[twoside,12pt]{article}
\usepackage[all, 2cell]{xy}
\usepackage{latexsym,amsmath, amsfonts, graphics, epsf, epic}
\usepackage{amsthm }
\usepackage{amssymb}
\usepackage{amsopn}
\usepackage{amscd}
\usepackage{color}
\usepackage[all]{xy}

\xyoption{poly}

\theoremstyle{plain}
\newtheorem{theorem}{Theorem}[section]

\newtheorem{lemma}[theorem]{Lemma}
\newtheorem{proposition}[theorem]{Proposition}

\theoremstyle{definition}

\newtheorem{definition}[theorem]{Definition}

\newtheorem{conjecture}[theorem]{Conjecture}

\begin{document}

\title{Growth for the central polynomials}

\maketitle

\centerline{Amitai Regev}
\medskip
\centerline{Dept. Math and CS}
\centerline{The Weizmann Institute,}
\centerline{Rehovot 76100, Israel}
\centerline{email: amitai.regev at weizmann.ac.il}

remark6.tex

\bigskip
{\bf Abstact}:
We study the growth of the central polynomials for the algebras $G$ and $M_k(F)$, the
infinite dimensional Grassmann algebra and the $k\times k$ matrices
over a field $F$ of characteristic zero.  In particular it follows that
$M_k(F)$ satisfy many proper central polynomials.

\bigskip


\medskip
\section{Introduction and main result}
Let $F$ be a field of characteristic zero, $A$  an $F$ algebra, $Id(A)\subseteq F\langle x \rangle$ the polynomial identities of $A$, and
let $Id^z(A)\subseteq F\langle x \rangle$ denote the  polynomials which are central on $A$:
$g(x_1,\ldots,x_n)\in Id^z(A)$ if for any $a_1,\ldots ,a_n\in A,$ $g(a_1,\ldots , a_n)\in center(A)$.
Such $g$ is {\it proper central} if $g$ is central and non-identity of $A$.
For example let $A=M_2(F)$, then $[x,y]^2\not \in Id(A)$ but $[x,y]^2 \in Id^z(A)$.
Note that polynomial identities
are considered here as central polynomials, and in particular $Id(A)\subseteq Id^z(A)$.
Clearly, if $g\in Id^z(A)$ then $g+Id(A)\subseteq Id^z(A)$. Thus the proper central polynomials of $A$
correspond to the space $Id^z(A)/Id(A)$.

The existence of proper central polynomials for matrix algebras was an important open problem in PI theory. It was
solved independently by Formanek~\cite{formanek.central} and by
Razmyslov~\cite{raz}, who constructed proper central polynomials for any matrix algebra $M_k(F)$.  However,
only very little is known about the question "how many proper central polynomials there are?"
We show here that this problem is related to
the growth of the central cocharacters and codimensions of the given algebra.

\medskip

In this paper we study that question in the cases $A=G$ and $A=M_k(F)$; here $G$ is the infinite dimensional Grassmann algebra. To study that "central growth" for a given algebra $A$ we intersect these spaces
with the multilinear polynomials $V_n=V_n(x_1,\ldots,x_n)$, then the proper multilinear central polynomials
correspond to the quotient space $D_n(A)=(V_n\cap Id^z(A))/(V_n\cap Id(A))$, with $\dim D_n(A)=\delta_n(A)$.
The sequence $\delta_n(A)$ determines the growth of the proper central polynomials.
The $S_n$ character of $D_n(A)$
is the difference of cocharacters (with corresponding codimensions)
$$
\chi_{S_n}(D_n(A))=\chi_n(A)-\chi_n^z(A)\qquad\mbox{hence}\qquad \delta_n(A)=c_n(A)-c_n^z(A).
$$
 In the first part of the paper we analyze  the case $A=G$ and determine, {\it precisely},
 these cocharacters, thus we determine precisely $\chi_{S_n}(D_n(A))$.
In the second part we study the case $A=M_k(F)$ and  prove that
$$
\lim_{n\to\infty} (\delta_n(M_k(F)))^{1/n}=k^2.
$$
This indicate that $M_k(F)$ satisfies a large (exponential) amount of proper central polynomials.

\section{Some generalities}
\subsection{Ordinary representations of $S_n$}

The representation theory of $S_n$ is a basic tool in what follows. Here is a brief review of that theory.
The irreducible characters of $S_n$ are indexed by the partitions $\lambda\vdash n$, denoted $\chi^\lambda$,
 with $f^\lambda=\deg \chi^\lambda$ the corresponding degree. To $\lambda\vdash n$ corresponds the
 (unique) Young diagram $D_\lambda$ (of shape $\lambda$) and (many) Young tableaux $T_\lambda$, again
 of shape $\lambda$. To such  tableaux $T_\lambda$ corresponds the semi-idempotent $e_{T_\lambda}$ in the
 group algebra $FS_n$; the left ideal $FS_ne_{T_\lambda}$ is an irreducible $S_n$-representation with
 corresponding $S_n$ character $\chi^\lambda$: $\chi_{S_n}(FS_ne_{T_\lambda})=\chi^\lambda$.
For a detailed account of that theory see for example~\cite{sagan}.

\subsection{Proper central polynomials}
Recall that the elements of $Id^z(A)/ Id(A)$ correspond to the {\it proper central polynomials of $A$}.
We would like to  estimate  the amount of the proper central polynomials of $A$, and we do that by
restricting to multilinear polynomials (see Definition~\ref{new1} below), since as in the case of identities,
the multilinear central polynomials generate all such polynomials.

\medskip
It is well known that $Id(A)$ is a $T$-ideal
in $F\langle x \rangle$, namely it is closed under substitutions and under ideal operations in
 $F\langle x \rangle$. On the other hand $Id^z(A)$ is only a $T$-subalgebra of  $F\langle x \rangle$,
  it is closed under substitutions and under algebra -- but not ideal -- operations in
 $F\langle x \rangle$.
Since $Id^z(A)$ and $Id(A)$ are closed under substitutions,
 the following definition makes sense as it is
independent of the particular variables $x_1,\ldots,x_n$.
\begin{definition}\label{new1}

Let
$V_n=V_n(x_1,\ldots,x_n)$ be the multilinear polynomials
of degree $n$ in $x_1,\ldots,x_n$. To study   $Id^z(A)/ Id(A)$
we intersect with $V_n$ since, as in the case of identities, the central polynomials are generated be the
multilinear central polynomials. We define

\medskip
1.~
$$
D_n(A):=\frac{V_n\cap Id^z(A)}{ V_n\cap Id(A)}\quad\mbox{and}\quad \delta_n(A):=\dim(D_n(A))=
\dim \left(\frac{(V_n\cap Id^z(A))}{ (V_n\cap Id(A))}\right)
$$
Similarly

\medskip
2.~~
$$
 c_n(A):=\dim \left(\frac{V_n}{V_n\cap Id(A)}\right) \qquad \chi_n(A):= \chi_{S_n}\left(\frac{V_n}{V_n\cap Id(A)}\right),
$$
$$
 c^z_n(A):=\dim \left(\frac{V_n}{V_n\cap Id^z(A)}\right) \qquad\mbox{and}\qquad \chi_n^z(A):= \chi_{S_n}\left(\frac{V_n}{V_n\cap Id^z(A)}\right),
$$

the ordinary and the central cocharacters with their corresponding codimensions.

\end{definition}
\begin{lemma}\label{new2}
$D_n(A)$ is a left $S_n$ module, its $S_n$ character is
\begin{eqnarray}\label{difference2}
\chi_{S_n}(D_n(A))=\chi_n(A)-\chi^z_n(A)
\end{eqnarray}
and by taking degrees we get
\begin{eqnarray}\label{difference3}
\delta_n(A)=c_n(A)-c^z_n(A).
\end{eqnarray}
\end{lemma}
\begin{proof}
Note that
$$
D_n(A)=\frac{V_n\cap Id^z(A)}{ V_n\cap Id(A)}\cong \frac{V_n/(V_n\cap Id(A))}{V_n/(V_n\cap Id^z(A))},
$$
isomorphism of quotient $S_n$ modules.  The proof follows by computing the corresponding $S_n$ characters
-- and their degrees.
\end{proof}

\section{The case $A=G$}

We now calculate the above invariants $c_n^z(G),$ $\chi_n^z(G)$ and $\delta_n(G)$ for $A=G$, the infinite dimensional Grassmann algebra.

\subsection{The results for $G$}
Denote $\chi_n^z(G)=\sum_{\lambda\vdash n}a_\lambda\chi^\lambda\quad\mbox{and}\quad
 \chi_n(G)=\sum_{\lambda\vdash n}b_\lambda\chi^\lambda.$ Since $\chi_n^z(G)\le \chi_n(G)$,
 all $a_\lambda\le b_\lambda.$
By~\cite{olsson} $ b_\lambda=1$ if $\lambda$ is a $(1,1)$ hook partition $\lambda=(n-j,1^j),$ and
$b_\lambda=0$ otherwise. It follows that $a_\lambda\in\{0,1\}$ if
$\lambda$ is a $(1,1)$ hook partition $\lambda=(n-j,1^j)$, and $a_\lambda=0$ otherwise. We prove here the following theorem.
\begin{theorem}\label{main}
Let $\chi_n^z(G)=\sum_{\lambda\vdash n}a_\lambda\chi^\lambda$, then $ a_{(n-2j,1^{2j})}=1$
for $0\le j\le\lfloor n/2\rfloor$, and $a_\lambda=0$ for all other partitions
$\lambda\vdash n$.
\end{theorem}
See  also Theorem~\ref{1.9} below.
\medskip

It is known~\cite{krakowsky} that $Id(G)=T([[x,y],u])$, the $T$-ideal of the triple commutator.
An obvious central polynomial for $G$ is the commutator
$[x,y]\in Id^z(G)$.
 Theorem~\ref{corollary2} below is of interest on its own; it says that the algebra of central polynomials of $G$
is $T$ generated by the commutator $[x,y]$ and by the triple commmutator $[[x,y],u]$.

\medskip

We remind here that the central polynomials form a $T$-subalgebra, but not an ideal,
hence they behave rather differently from identities. For example,
modulo $Id^z(G)$ the monomials $x_1x_2$ and $x_2x_1$ are linearly dependent
(since  $x_1x_2-x_2x_1\in Id^z(G)$) but
$x_3x_1x_2$ and $x_3x_2x_1$ are linearly independent modulo $Id^z(G)$, see Lemma~\ref{cancellation2} below.

\subsection{Some preparations}
Our aim here is to calculate the $S_n$ character $\chi^z_n(G)$, the $n$-th central cocharacter of $G$.
We first calculate the restriction $\chi^z_n(G)\downarrow_{S_{n-1}}$, proving it in fact is the ordinary
$n-1$ cocharacter of $G$, see Theorem~\ref{corollary1}.
 A main tool here is the following isomorphism.

\begin{proposition}
There is a canonical isomorphism of $FS_{n-1}$ modules
$$
\frac{V_n}{V_n\cap Id^z(G)}\cong \frac{V_{n-1}}{V_{n-1}\cap Id(G)}.
$$
It follows that the cocharacters satisfy
$$
\chi_n^z(G)\downarrow_{S_{n-1}}=\chi_{n-1}(G),
$$
{and in particular the codimensions satisfy
$c_n^z(G)=c_{n-1}(G).$}
Since $c_n(G)=2^{n-1}$~\cite{krakowsky}, hence $$c_n^z(G)=2^{n-2}=\frac{1}{2}c_n(G).$$
\end{proposition}
The proof is given in Section~\ref{isomorphism} below.


\begin{lemma}\label{cancellation}
Given $h=h_n\in V_n$, there exist $h_{n-1}\in V_{n-1}$ (given canonically, so the
map $h_n\to h_{n-1}$ is well defined) such that $h_n=x_nh_{n-1}+q_n$ where $q_n\in Id^z(G)\cap V_n$.
\end{lemma}
\begin{proof}
Suffices to prove when $h\in V_n$ is a monomial $h=M$. Then $M=ax_nb$, so $M=x_nba+[a,x_nb]$,
and $q_n=[a,x_nb]\in  Id^z(G)\cap V_n. $
\end{proof}
{\bf Note} that this presentation is not unique, since if $h'_{n-1}\in V_{n-1}\cap Id(G)$, we have
$h_n=x_n(h_{n-1}+h'_{n-1})+(q_n-x_nh'_{n-1}),$  and $q_n-x_nh'_{n-1}\in Id^z(G)\cap V_n.$

\medskip

\begin{definition}
Let $B=\{e_{i_1}\cdots e_{i_r}\mid r=1,2,\ldots,~~i_1<\cdots< i_r \}$ the canonical linear basis of $G$.
 We consider
{\it disjoint} substitutions, namely $x_j\to \bar x_j\in B$, $j=1,\ldots, n$ such that
$\bar x_1\cdots \bar x_n\ne 0$. Given $h_n\in V_n$, such a substitution induces
$h_n=h_n(x_1,\ldots,x_n)\to \bar h_n=h_n(\bar x_1,\ldots ,\bar x_n)$.
If $e_{i_1}\cdots e_{i_r}\ne 0$, denote $\ell(e_{i_1}\cdots e_{i_r})=r$ (length). Such a substitution is even (odd) if
$\ell(\bar x_1\cdots\bar x_n)=\sum _{i=1}^n\ell(\bar x_i)$
is even (odd).
\end{definition}
{\bf Note} that such a disjoint substitution is even if and only if $\bar x_1\cdots\bar x_n$ is central in $G$.

\begin{lemma}\label{cancellation2}
Let $h_{n-1}\in V_{n-1}$ then $x_nh_{n-1}\in Id^z(G)$ if and only if
$h_{n-1}\in Id(G)$.
\end{lemma}
\begin{proof}
1. ~Assume $h_{n-1}\in Id(G)$ then $x_nh_{n-1}\in Id(G)\subseteq Id^z(G).$ Conversely,

\medskip
2. ~Let $x_nh_{n-1}\in Id^z(G).$ Make a disjoint substitution $x_j\to \bar x_j$, $j=1,\ldots,n$,
then $\bar x_n\bar h_{n-1}$ is central, hence even. Now slightly change that substitution: leave
$x_j\to \bar x_j$, $j=1,\ldots,n-1$ unchanged, and let $x_n\to \hat x_n $ with the parity of $\hat x_n $
opposite that of $\bar x_n .$ Since $x_nh_{n-1}\in Id^z(G),$ the result is still even, hence we
must have $\bar h_{n-1}=0$, namely $h_{n-1}\in Id(G)$.
\end{proof}

\subsubsection{An isomorphism}\label{isomorphism}
\begin{lemma}

Let $$\varphi:V_{n-1}\longrightarrow \frac{V_n}{V_n\cap Id^z(G)}$$
be given as follows: Given $h_{n-1}\in V_{n-1}$, then
$\varphi (h_{n-1})=x_nh_{n-1} + (V_n\cap Id^z(G)).$
Then

\medskip

1.  $\varphi$ is onto  ${V_n}/({V_n\cap Id^z(G)})$, and

\medskip
2.  $\ker\varphi = V_{n-1}\cap Id(G)$.
\end{lemma}

\begin{proof}
1.  Let $y_n\in V_n$ then by Lemma~\ref{cancellation}
we can write $y_n=x_ny_{n-1}+q_n$, ~$q_n\in Id^z(G)\cap V_n.$
Thus $\varphi( y_{n-1})= x_ny_{n-1}+(V_n\cap Id^z(G))= x_ny_{n-1}+q_n+(V_n\cap Id^z(G))=y_n+(V_n\cap Id^z(G)).$

\bigskip

2. Let $h_{n-1}\in V_{n-1}$ then $h_{n-1}\in ker \varphi$ if and only if

\medskip $x_n h_{n-1}\in V_n\cap Id^z(G)$, so  if and only if

\medskip
$x_n h_{n-1}\in Id^z(G),$ by Lemma~\ref{cancellation2} if and only if

\medskip
$ h_{n-1}\in Id(G)$ if and only if

\medskip
$ h_{n-1}\in V_{n-1}\cap Id(G).$

\end{proof}

As a corollary we have

\begin{theorem}\label{corollary1}
$\varphi$ induces the isomorphism
$$
\bar\varphi: \frac{V_{n-1}}{V_{n-1}\cap Id(G)}\cong\frac{V_{n}}{V_{n}\cap Id^z(G)}
$$
an isomorphism of $FS_{n-1}$ modules.
Therefore
\begin{eqnarray}\label{ab2}
\chi_{n-1}(G)=\chi_n^z(G)\downarrow_{S_{n-1}}.
\end{eqnarray}
Recall~\cite{krakowsky} that $c_n(G)=2^{n-1}$. Thus in particular, the codimensions satisfy
$$
c_n^z(G)=c_{n-1}(G)=2^{n-2}=\frac{1}{2}2^{n-1}=\frac{1}{2}c_n(G).
$$
\end{theorem}

\subsubsection{The central cocharacter $\chi^z_n(G)$}

We can now prove our main result of this section.
\begin{theorem}\label{1.9}
$$
\chi^z_n(G)=\sum_{j=0}^{\lfloor n/2\rfloor} \chi^{(n-2j,1^{2j})}
=\chi^{(n)}+\chi^{(n-2,1^2)}+\chi^{(n-4,1^4)}+\cdots
$$
then by Equation~\eqref{difference2} it follows that
$$
\chi_{S_n}(D_n(G))=
\sum_{j=0}^{\lfloor n/2\rfloor} \chi^{(n-2j+1,1^{2j-1})}
=\chi^{(n-1,1)}+\chi^{(n-3,1^3)}+\chi^{(n-5,1^5)}+\cdots.
$$

\end{theorem}
\begin{proof}
We know~\cite{olsson} that
\begin{eqnarray}\label{equation}\chi_n(G)=\sum_{k}\chi^{(n-k,1^k)}.
\end{eqnarray}
and
\begin{eqnarray}\label{end}
\chi^z_n(G)=\sum_{k\ge 0} a_k\cdot \chi^{(n-k,1^k)}
\end{eqnarray}
 where each multiplicity $a_k$ is either $=0$ or $=1$.

\medskip
The trivial character $\chi^{(n)}$ corresponds to the semi-idempotent $e_{(n)}=\sum_{\sigma\in S_n}\sigma$
and clearly  $e_{(n)}$ is not a central identity of $G$, hence $a_0=1$.
 By Equations~\eqref{ab2} and~\eqref{equation},
  $$\chi^z_n(G)\downarrow_{S_{n-1}}=\chi_{n-1}(G)=\sum_{j\ge 0}\chi^{(n-1-j,1^j)}.
  $$
Restrict now the $S_n$ character $\chi^z_n(G)$ down to $S_{n-1}$. By "branching", $\chi^{(n)}\downarrow _{S_{n-1}}=\chi^{(n-1)}$,
while for $k>0$
$$\chi^{(n-k,1^k)}\downarrow _{S_{n-1}}=\chi^{(n-k-1,1^k)}+\chi^{(n-k,1^{k-1})}.$$
   Together with~\eqref{end} we  obtain the following trivial system of equations for the multiplicities $a_k$:
$$
a_0=1, ~~a_0+a_1=1, ~~a_1+a_2=1, ~~a_2+a_3=1 \ldots
$$
and the proof follows.

\end{proof}

\subsubsection{$T$ generation of $Id^z(G)$}
We assume now that the characteristic of the base field is zero.
\begin{theorem}\label{corollary2}
The algebra $Id^z(G)$ of the central polynomials of $G$
contains $Id(G)$, and is $T$-generated as an $F$-algebra over $Id(G)$  by the commutator $[x,y]$.
In other words, the algebra $Id^z(G)$ of the central polynomials of $G$ is $T$-generated by $[[x,y],u]$
and by $[x,y]$.
\end{theorem}
\begin{proof}
Let $H\subseteq F\langle x \rangle$ be the subalgebra containing $Id(G)$, and $H$ is $T$-generated, as an $F$-algebra over $Id(G)$,  by the commutator $[x,y]$.
Thus $H\subseteq Id^z(G)$. Now the previous arguments all work with $H$ replacing $Id^z(G)$. In particular
$$
 \frac{V_{n-1}}{V_{n-1}\cap Id(G)}\cong\frac{V_{n}}{V_{n}\cap H}
$$
for all $n$. This implies that for all $n$, $\dim (V_n\cap H)=\dim (V_n\cap Id^z(G))$
and since $H\subseteq Id^z(G)$ it follows that for all $n$,  $V_n\cap H=V_n\cap Id^z(G)$. Since we assume that $char(F)=0$, this implies that $H=Id^z(G)$.

\end{proof}

In fact, the same argument shows that in any characteristic, the multilinear central polynomials of $G$
are  $T$-generated by $[[x,y],u]$ and by $[x,y]$.

\medskip
Motivated by Theorem~\ref{corollary2} we make the following conjecture.
\begin{conjecture}
In characteristic zero the central polynomials of a PI algebra $A$ satisfy the Specht property -- in the following sense:
the sub algebra $Id^z(A)\subseteq  F\langle X\rangle$ of the central polynomials of $A$
is $T$-generated {\it over} $Id(A)$ by a finite
set of central polynomials.
\end{conjecture}

\section{The case $A=M_k(F)$}

\subsection{Further introductory remarks}
We are interested in the question
how many proper central polynomials the algebra $M_k(F)$ satisfies, namely,
how large is $\delta_n(M_k(F))$, and what is
the asymptotics  of $\delta_n(M_k(F))$ as $n$ goes to infinity. The case $k=1$ being trivial,
we assume  w.l.o.g here that $k\ge 2$.

\medskip
Lots of work was done on the invariants  $c_n(M_k(F))$ and $\chi_n(M_k(F))$. In particular it was proved that
\begin{eqnarray}\label{bliz}
\lim_{n\to\infty} c_n(M_k(F))^{1/n} = k^2,
\end{eqnarray}
see~\cite{regev3},~\cite[Theorem 5.10.2]{G.Z}.

\medskip
Turn now to central polynomials.
By~\cite{formanek.central}~\cite{raz} (see also~\cite{formanek.regev})  $M_k(F)$ satisfies proper central polynomials, hence
  $c_n^z(M_k(F))$ is strictly smaller than $ c_n(M_k(F))$.   Recall that
$$
D_n=\frac{V_n\cap Id^z(M_k(F)))}{ V_n\cap Id(M_k(F)))}\quad\mbox{and}\quad
\delta_n(M_k(F))=\dim (V_n\cap Id^z(M_k(F))) -\dim (V_n\cap Id(M_k(F))).
$$
The next theorem is the main result of this section. It
indicates that
$M_k(F)$ satisfies many proper central polynomials.

\begin{theorem}\label{main2}
Let $k\ge 2$
then (compare with~\eqref{bliz})
$$
\lim_{n\to\infty}(\delta_n(M_k(F)))^{1/n}=k^2.
$$
\end{theorem}




\subsection{Towards the proof of Theorem~\ref{main2}}
\subsubsection{A sandwich for $\delta_n(A)$}

Lemma~\ref{new2} already implies the upper bound
\begin{eqnarray}\label{upper.bound}
\delta_n(A)\le c_n(A).
\end{eqnarray}
 For the lower bound we need Proposition~\ref{shem.1} below, which we now prove.

\begin{lemma}\label{sham}
Let $H_n\subseteq V_n$ be a subspace satisfying $H_n\cap Id^z(A)=0$. Then $H_n$ imbeds in $D_n(A)$,
and in particular $\dim H_n\le dim D_n(A)=\delta_n(A)$.
\end{lemma}
\begin{proof}
Map $\varphi:H_n\to D_n$ as follows:

$$
\varphi:V_n\longrightarrow ~\frac{V_n}{V_n\cap Id(A)}~\longrightarrow
\frac{V_n/V_n\cap Id(A)}{V_n/V_n\cap Id^z(A)}\cong D_n(A)
$$
via ~$x\in V_n$,~$~~\varphi:x\longrightarrow x+{V_n\cap Id(A)} \longrightarrow x+{V_n\cap Id^z(A)}=\varphi(x),$ well defined maps.
\medskip
If $x\in H_n$ and $\varphi(x)=0$
then $x\in {V_n\cap Id^z(A)}$ so $x=0$ since then
$x\in H_n\cap Id^z(A)=0$.
\end{proof}
Recall the identification $V_n=FS_n$.
\begin{lemma}\label{18}
Let $g\in V_n,$ $g\not\in Id^z(A)$ and assume $FS_ng$ is an irreducible left $S_n$ module. Then $FS_ng\cap Id^z(A)=0$.
\end{lemma}
\begin{proof}
Let $h\in FS_ng\cap Id^z(A)$. If $h\not = 0$ then $FS_nh=FS_ng$ (by irreducibility). Since
by assumption $h\in Id^z(A)$
we get $g\in FS_nh\subseteq Id^z(A)$ so $g\in Id^z(A)$, contradiction.
\end{proof}

The following proposition is one of the main tools for estimating $\delta_n(A)$. Of course, for the upper bound
we already have~\eqref{upper.bound}.
and as usual, the difficulties are with the lower bound.

\begin{proposition}\label{shem.1}
Let $\lambda\vdash n$ with a corresponding tableau $T_\lambda$ and semi-idempotent $e_{T_\lambda}$, and assume that $e_{T_\lambda}\not\in Id^z(A)$, then $f^\lambda\le \delta_n(A)$.
\end{proposition}
\begin{proof}
Let $FS_ne_{T_\lambda}=H_n,$ then $H_n\cap Id^z(A)=0$, so by Lemma~\ref{sham}
$f^\lambda =\dim H_n\le dim D_n(A)$.
\end{proof}

Recall from Lemma~\ref{new2} that
$\delta_n(M_k(F))=c_n(M_k(F))-c_n^z(M_k(F))$
and in particular\\ $\delta_n(M_k(F))\le c_n(M_k(F)).$ Together with Equation~\eqref{bliz} this proves
\begin{lemma}
$$
\lim_{n\to\infty} \delta_n(M_k(F))^{1/n} \le\lim_{n\to\infty} c_n(M_k(F))^{1/n} = k^2.
$$
\end{lemma}

A lower bound is given in the following Lemma.

\begin{lemma}\label{lower}
We also have
$$
k^2\le \lim_{n\to\infty} \delta_n(M_k(F))^{1/n}.
$$
\end{lemma}

{\bf Idea of the proof:}  For each $n$ we construct a partition $\lambda\vdash n$ with a tableau
$T_\lambda$ such that $e_{T_\lambda}$ is non central;  $\lambda$ is constructed such that
$k^2\le \lim_{n\to\infty}(f^\lambda)^{1/n}$.
Let $H_n=FS_ne_{T_\lambda}$, then
$H_n\cap Id^z(M_k(F))=0$. By Lemma~\ref{sham} $f^\lambda \le \delta_n(M_k(F))$, and the proof is complete.
That proof takes the rest of this paper.

\subsubsection{Gluing tableaux}

The proof of Lemma~\ref{lower} applies several ingredients which we briefly
review. We start with "gluing together" Young tableaux, see~\cite[Theorem 1.6]{regev1}.
 \medskip
 Let $\lambda=(a_1,\ldots, a_r)$ be a partition, with $a_1\ge\cdots\ge a_r \ge 1$, so $\ell(\lambda)=r$.
 Recall that $\lambda$ is identified with its Young diagram $D_\lambda$.
 Let $h_j(\lambda)=\lambda'_j$ denote the length of the $j$-th column of $\lambda$, so $h_1(\lambda)=\ell(\lambda)$,
 and $h_{a_1}(\lambda)$ is the length of the last (rightmost) column of $\lambda$.
Let $\mu=(b_1,\ldots,b_s)$ be a second partition, then denote
$\lambda*\mu:=(\lambda_1+\mu_1,\lambda_2+\mu_2,\ldots )\vdash |\lambda|+|\mu|$.
If $h_{a_1}(\lambda)\ge h_1(\mu)$ then $D_\mu$ can be glued to the right of $D_\lambda$ and in that case
we denote $D_{\lambda*\mu}=D_\lambda | D_\mu$.

\medskip
  A tableau $T_\lambda$ of shape $\lambda$
 and with entries $t_{i,j}$ is denoted $T_\lambda=D_\lambda(t_{i,j})$. Also let $T_\mu=D_\mu(u_{i,j})$
 be a second tableau of shape $\mu$, and assume $D_{\lambda*\mu}=D_\lambda | D_\mu$. Then glue $T_\lambda$ with
 $T_{\mu}$ by constructing $T_{\lambda*\mu}=D_\lambda(t_{i,j})\mid D_\mu(u_{i,j}+|\lambda|)$.

\medskip
As usual the tableau $T_\lambda$ corresponds to the semi-idempotent $e_{T_\lambda}$ in $FS_{|\lambda|}$.
We have
\begin{theorem}\label{1.6}~\cite[Theorem 1.6]{regev1}
Let $\lambda=(a_1,\ldots,a_r)\vdash m$, $\mu\vdash n$ with corresponding tableaux $T_\lambda$ and $ T_\mu$ and
assume $h_{a_1}(\lambda)\ge h_1(\mu)$, namely $D_\lambda$ and $ D_\mu$ can be glued to form
$D_{\lambda*\mu}=D_\lambda\mid D_\mu$.
Let
$T_{\lambda*\mu}=D_\lambda(t_{i,j})\mid D_\mu(u_{i,j}+m)=
D_{\lambda\mid \mu}(w_{i,j})$ and substitute $w_{i,j}\to y_i$. Then
there exists $d\in\mathbb N$ such that
$$
e_{T_{\lambda*\mu}}(y)=d\cdot e_{T_\lambda}(y)\cdot e_{T_\mu}(y).
$$

\end{theorem}

\subsection{Capelli-type polynomials}\label{regev}
Another main ingredient in the proof of Lemma~\ref{lower} are the Capelli-type polynomials, called sometime Regev-polynomials~\cite{formanek.regev},~\cite{GV}.
 These polynomials correspond to
rectangles of height $k^2$. The case of the $k^2\times 2$ rectangle (namely the partition
$(2^{(k^2)})$) is done in details in~\cite{formanek.regev}.

\subsubsection{Two sets of variables}\label{two}

 In the case of two sets of variables the polynomial is constructed
as follows: first construct the monomial
\begin{eqnarray}\label{monomial2}
M(x,y)=(x)(y)(xxx)(yyy)(xxxxx)(yyyyy)\cdots
\end{eqnarray}
of pairs of blocks of the odd lengths $1,3,5,\ldots ,2k-1$, then alternate the $x'$s and independently
alternate the $y'$s. Thus the corresponding polynomial is
$$
L(x,y)=\sum_{\sigma\in S_{k^2}}\sum_{\tau\in S_{k^2}}sgn (\sigma)sgn(\tau)M(x_{\sigma},y_{\tau}),
$$
in two sets of $k^2$ alternating variables. It corresponds to a particular tableau on
the $k^2\times 2$ rectangle, namely
the partition $(2^{(k^2)})$.
By~\cite{formanek.regev} $L(x,y)$ is a proper central polynomial
of $M_k(F)$.
It also satisfies

\medskip

{\bf Property L}: For any bijection $x_u\leftrightarrow \bar x_u =e_{i,j}$
between the set $\{x_1,\ldots,x_{k^2}\}$ and the set $\{e_{i,j}\mid 1\le i,j\le k\},$ and similarly
$y_v\leftrightarrow \bar y_v= e_{i,j}$
we get, up to a $\pm$ sign, the same {\it non zero} scalar value $L(\bar x, \bar y)$.

\medskip
Equating $x_i=y_i, ~i=1,\ldots,k^2$ we obtain $g_2(x)=L(x,x)=g_2(x_1,\ldots,x_{k^2})$, and {\it Property L}
implies {\it Property G:}

\medskip
{\bf Property G:} For any bijection $x_u\leftrightarrow \bar x_u =e_{i,j}$
we get, up to a $\pm$ sign, the same non zero scalar value $g_2(\bar x)$.

\medskip
{\bf Note} that $g_2(x)$ is of degree 2 in each $x_i$. Its multilinearization yields
$d\cdot e_{T_\nu}\in S_{2k^2}$, where $d\ne 0$, $\nu = (2^{(k^2)})$ and $T_\nu$ is is the tableau
of shape $\nu$
corresponding to the monomial~\eqref{monomial2}.  $g_2(x)$ is proper central, therefore so is
 $e_{T_\nu}$.

\subsubsection{Three sets of variables}
Instead of two sets of variables we can repeat that construction with, say, three sets of $k^2$ alternating variables, starting with the monomial
\begin{eqnarray}\label{monomial3}
M(x,y,z)=(x)(y)(z)(xxx)(yyy)(zzz)(xxxxx)(yyyyy)(zzzzz)\cdots
\end{eqnarray}
It corresponds to
a particular tableau on
the $k^2\times 3$ rectangle,
namely on the partition $(3^{(k^2)})$. Similar to Section~\ref{two}, the resulting polynomial $L(x,y,z)$, in three sets of $k^2$ alternating variables is proper central for $M_k(F)$.
Equating $x_i=y_i=z_i, ~i=1,\ldots,k^2$ we obtain $g_3(x)=L(x,x,x)=g_3(x_1,\ldots,x_{k^2})$, and
it satisfies the analogue properties L and G in three sets of $k^2$ alternating variables.

\subsection{The proof of Lemma~\ref{lower} and Theorem~\ref{main2} }
\begin{proof}
We need to show that
$$
k^2\le \lim_{n\to\infty} \delta_n(M_k(F))^{1/n}.
$$
Let $n\ge 2k^2$ (this is not a restriction since we later send $n$ to infinity) and
 write it as $n=k^2m+r$ where $0\le r<k^2$, so $m\ge 2$.

\medskip
Case 1: $m$ is odd. In that case $m\ge 3$ so $m=2q+3$ where $q\ge 0$.

\medskip

Case 2: $m=2q$ is even, $q\ge 1$.

\medskip
The summand $k^2m$ of $n$ determines the $k^2\times m$ rectangle.
In case 1,  $k^2m=k^22q+k^23$ so the $k^2\times m$ rectangle is obtained by gluing $q$
$~k^2\times 2 $ rectangles, plus the single $k^2\times 3 $ rectangle.
In case 2, $k^2m=k^22q$ so the $k^2\times m$ rectangle is obtained by gluing $q$
$~k^2\times 2$ rectangles. Each $k^2\times 2$ rectangle,  with its particular tableau,
 corresponds to $g_2(x)$ of Section~\ref{regev}. Similarly
the $k^2\times 3 $ rectangle with  its particular tableau corresponds to $g_3(x)$.

\medskip


 In case 1, by the gluing technique we glue together $q$ such tableaux each of shape
 $(2^{(k^2)})$, together with the single tableau os shape $(3^{(k^2)})$.
The corresponding shape is  $(m^{(k^2)})$ -- with the corresponding polynomial $g_2(x)^q\cdot g_3(x)$,
 which is proper central of $M_k(F)$, since both $g_2(x)$ and $g_3(x)$ are such.

 \medskip
 In case 2 we glue together $q$ such tableaux each of shape
 $(2^{(k^2)})$.
The corresponding shape is  $((2q)^{(k^2)})$ -- with the corresponding polynomial $g_2(x)^q$,
 which again is proper central of $M_k(F)$.

\medskip
Turn now to the integer $0\le r< k^2$. Construct a partition $\pi\vdash r$
together with a tableau $T_\pi$ of shape $\pi$ and with a corresponding polynomial
$h_\pi=h_{T_\pi}$, such that $h_\pi(x)$ is {\it non} central for $M_k(F)$.  Indeed if $r=0$ then we are
already done, so assume $r\ge 1$. Then, for example, choose $\pi=(r)$, $T_\pi$ is the one-row tableau
filled with $1,\ldots,r$. Then $h_\pi(x)=\sum_{\sigma\in S_r}x_{\sigma (1)}\cdots x_{\sigma(r)}$
which is non central for $M_k(F)$: for example for $i=1,\ldots,r$ substitute $x_i\to\tilde x_i=diagonal(1,2,\ldots,k)$,
then $h_\pi(\tilde x)=r!\cdot diagonal(1^r,2^r,\ldots,k^r),$ which is non central since $k\ge 2$.

\medskip
Since $h_\pi(x)$ is multilinear and of degree $r<k^2$, there exist an injection
$$
x_d\to \bar x_d \in \{e_{i,j}\mid 1\le i,j\le k\}\quad d=1,\ldots,r
$$
such that $h_\pi(\bar x)$ is
non central. Complete this substitution to a bijection
 $x_d\longleftrightarrow\bar x_d=e_{i,j}$ of $x_1,\ldots,x_{k^2}$
with the $k^2$ elements $\{e_{i,j}\mid 1\le i,j\le k\}$,  so $g_2(\bar x)$ and $g_3(\bar x)$
 are proper central
and $h_\pi(\bar x)$ is non-central -- of $M_k(F)$.

\medskip
 Glue now  $T_\pi$ to the right of $T_\mu$; this can be done since $\ell(\pi)\le k^2$,
and we obtain the tableau $T_\mu *T_\pi$.
  In case 1  that tableau $ T_\mu *T_\pi$ corresponds to
the polynomial $g_2(x)^q\cdot g_3(x)\cdot h_\pi(x)$. The
shape of  $ T_\mu *T_\pi$ is the partition
 $$
 \lambda=\mu *\pi=((2q+3)^{(k^2)}) *\pi =(2q+3+\pi_1,\ldots,2q+3+\pi_{k^2})\vdash (2q+3)k^2 +r=n.
 $$
Similarly in case 2
$$\lambda=\mu *\pi=((2q)^{(k^2)}) *\pi =(2q+\pi_1,\ldots,2q+\pi_{k^2})\vdash 2qk^2 +r=n.$$

 Since ~{\it proper-central} ~multiplied by ~{\it non-central} ~is ~{\it non-central},~
 hence in case 1 \\$g_2(x)^q\cdot g_3(x)\cdot h_\pi(x)$ is
a {\it non central} polynomial of $M_k(F)$; it corresponds to $\lambda$, with the corresponding tableau $T_\lambda$,
therefore $e_{T_\lambda}$ is non central. By Proposition~\ref{shem.1} this implies that
$\delta _n(M_k(F))\ge f^\lambda $.
Similarly in case 2.

\medskip
Now send $n$ (hence also $m$ and $q$) to infinity, then  $\mu$ becomes a large rectangle (of fixed height $k^2$) and
 $\pi\vdash r$, $0\le r<k^2$, is a small extra part of  $\lambda=\mu*\pi$. For the $k^2\times m$ rectangle $\mu=(m^{(k^2)})$
it is known that
$$\lim_{m\to \infty}(f^\mu)^{1/|\mu|}=  \lim_{|\mu|\to \infty}(f^\mu)^{1/|\mu|}=k^2   $$
see for example~\cite[Section 3]{explicit}.
 Also, by "Branching", $f^\lambda\ge f^\mu$ since $\lambda\supseteq \mu$.
Hence as $n,$  $m$ and $q$ go to infinity,
$$
 (\delta_n(M_k(F)))^{1/n}    \ge (f^\lambda)^{1/n}\ge (f^\mu)^{1/n}=((f^\mu)^{1/|\mu|})^{|\mu| /n  }\to (k^2)^{|\mu| /n}\to k^2
$$
since $|\mu|/n\to 1.$
The proof of Lemma~\ref{lower} now follows.
\end{proof}

\subsection{Some conjectures}
\begin{conjecture}\label{C1}
For any integer $k\in\mathbb N$ there exist a real number $0\le \alpha_k \le 1$ such that
\begin{eqnarray}\label{ratio}
\lim_{n\to\infty} \frac{c_n^z(M_k(F))}{c_n(M_k(F))} =\alpha_k.
\end{eqnarray}
\end{conjecture}

More generally, an analogue of~\eqref{ratio} holds when $M_k(F)$ is replaced by a more
general algebra $A$.

\medskip
\begin{conjecture}
Giambruno and Zaicev~\cite{G.Z2},~\cite{G.Z} proved that for any PI algebra $A$, when $n$ goes to infinity
$$
c_n(A)\sim \alpha\cdot n^g\cdot d^n
$$
where $d$ is an integer. In addition Berele~\cite{berele} proved that $g\in \frac{1}{2}\mathbb Z$.

\medskip
We tend to conjecture that the analogue theorems hold for $\delta_n(M_k(F))$,
and probably for other algebras $A$.

\medskip
\end{conjecture}




\end{document}